\documentclass[11pt,a4paper]{amsart}
\usepackage[textheight=22.5cm,textwidth=15cm,centering]{geometry}

\usepackage{hyperref}
\hypersetup{colorlinks=true,linkcolor=blue,citecolor=blue}    
\usepackage{graphicx}

\usepackage{xcolor}
\usepackage{verbatim}
\usepackage{amsmath}
\usepackage{amssymb, mathrsfs}
\usepackage{amsbsy}

\usepackage{amscd}

\usepackage[english]{babel}
\usepackage{todonotes}
\usepackage[T1]{fontenc}
\usepackage{amsthm}

\usepackage{color}
\allowdisplaybreaks

\newcommand{\R}{\mathbb{R}}

\newcommand{\e}{\varepsilon}

\newcommand{\Z}{\mathbb{Z}}

\renewcommand{\P}{\mathbb{P}}

\newcommand{\dd}{\mathrm{d}}

\newcommand{\lin}{\left[\kern-0.15em\left[}
\newcommand{\rin} {\right]\kern-0.15em\right]}
\newcommand{\linf}{[\kern-0.15em [}
\newcommand{\rinf} {]\kern-0.15em ]}
\newcommand{\ilin}{\left]\kern-0.15em\left]}
\newcommand{\irin} {\right[\kern-0.15em\right[}

\def\ben#1{\begin{equation}#1\end{equation}}

\def\al#1{\begin{align*}#1\end{align*}}
\def\aln#1{\begin{align}#1\end{align}}

\renewcommand{\tilde}{\widetilde}

\newtheorem{lem}{Lemma}[section]
\newtheorem{remark}[lem]{Remark}

\newtheorem{thm}[lem]{Theorem}

\newtheorem{cor}[lem]{Corollary}

\newtheorem {defin}[lem] {Definition}
\newtheorem {rem}[lem] {Remark}

\newcounter{assu}
\DeclareMathOperator{\diverg}{div}

\usepackage{color}

\numberwithin{equation}{section}

\title[Asymptotics of the $p$-capacity in the critical regime]
{Asymptotics of the $p$-capacity in the critical regime}
\date{\today}
\author{Cl\'ement Cosco} 
\address[Cl\'ement Cosco]
{Weizmann Institute of Science, Israel}
\email{clement.cosco@gmail.com}

\author{Shuta Nakajima} 
\address[Shuta Nakajima]
{University of Basel, Basel, Switzerland}
\email{shuta.nakajima@unibas.ch}

\author{Florian Schweiger} 
\address[Florian Schweiger]
{Weizmann Institute of Science, Israel}
\email{florian.schweiger@weizmann.ac.il}


\begin{document}

\begin{abstract} 
In this note, we are interested in the asymptotics as $n\to\infty$ of the $p$-capacity between the origin and the set $nB$, where $B$ is the boundary of the unit ball of the lattice $\mathbb Z^d$. The $p$-capacity is defined as the minimum of the Dirichlet energy $\frac{1}{2}\sum_{x\in \mathbb Z^d} \sum_{y\sim x} |f(x)-f(y)|^{p}$ with $f$ subject to the boundary conditions $f(0)=0$ and $f\geq 1$ on $nB$. This variational problem has arisen in particular in the study of large deviations for first passage percolation.
For $p<d$, the $p$-capacity converges to some positive constant, while for $p>d$ the capacity vanishes polynomially fast. The present paper deals with the case $p=d$, for which we prove that the $p$-capacity vanishes as $c_d (\log n)^{-d+1}$ with an explicit constant $c_d$. 
\end{abstract}

\maketitle

\section{Introduction}
\subsection{Setting}
Fix $p\geq 1$. Let $G=(V,E)$ be a graph and $A,B\subset V$.  Define the $p$-\emph{capacity} between two sets $A$ and $B$ as:
 \begin{equation}
 \label{general lambda}
\begin{aligned}
  & {\rm Cap}_{G,p}(A;B)\\
   &=\inf_{f: V\to \R}\left\{\sum_{e=\langle x, y\rangle \in E} |f(x)-f(y)|^{p}\ \middle|~\,f(x)\geq 1~\forall x\in B,\,f(x)=0~\forall x\in A\right\}.
   \end{aligned}
\end{equation}
Let $D_M=\{x\in \Z^d:~|x|_\infty \leq M\}$ be the ball of radius $M$, $\partial D_M=\{y\in \mathbb Z^d, |y|_\infty = M\}$ its boundary, and $E_M$ the set of all edges whose both ends are in $D_M$. We consider the graph $G_n=(D_n,E_n)$ and define: 
\[{\rm Cap}_{d,p}(n)={\rm Cap}_{G_n,p}(\{0\},\partial D_n).\]
In \cite{CN}, the first two authors have shown the following:
\begin{thm}[Theorem 2.2 in \cite{CN}]\label{asymptotic for lambda general}
  For $d\geq 1$ and $p>0$, the function
\[
\kappa_{d,p}(n)=
\begin{cases}
  {\rm Cap}_{d,p}(n)&\text{ if }p<d,\\
 (\log{n})^{d-1}  {\rm Cap}_{d,p}&\text{ if }p=d,\\
  n^{p-d}  {\rm Cap}_{d,p}&\text{ if }p>d.
\end{cases}\]  
   is bounded away from $0$ and $\infty$.
\end{thm}
In this {article}, we give a precise estimate in the critical case $p=d$.

\subsection{Main result} 
\begin{thm} \label{th:mainTh} Assume that $p=d$. As $n\to\infty$,
\begin{equation} \label{eq:asymptoticCap}
{\rm Cap}_{d,d}(n) \sim \frac{{\bf \rm vol}_{d-1}(S^{d-1}_q)}{(\log n)^{d-1}},
 \end{equation}
 where $q=\frac{d}{d-1}$, $S_q^{d-1}=\{x\in \R^d:~ |x|_q=1\}$ and ${\bf \rm vol}_{d-1}(S_q^{d-1})$ denotes the $(d-1)$-dimensional volume of $S_q^{d-1}$.
\end{thm}
\begin{remark}
One can replace the norm in the definition of $D_M$ by any norm $|\cdot|_r$ with $r>0$ and find the same asymptotics as in \eqref{eq:asymptoticCap} with the same proof. We keep the norm $|\cdot |_\infty$ for simplicity.
\end{remark}
\subsection{Background and related works}
Asymptotics of the $p$-capacity have appeared naturally in the study of upper tail large deviations in the context of first passage percolation on $\mathbb Z^d$. Let us briefly describe the model. Let $(\tau_e)_{e\in E}$ be i.i.d.\ weights defined on the edges of $\mathbb Z^d$, and 
\[{\rm T}(0,x) = \inf_{\gamma : 0\to x} \sum_{e\in \gamma} \tau_e,\]
be the infimum of the weight of all paths from $0$ to $x\in\mathbb Z^d$, that we call the \emph{passage time}. The passage time satisfies the following law of large numbers:
\[\mu(x)=\lim_{n\to\infty} \frac{1}{n} {\rm T}(0,nx)\,\text{a.s.},\]
where $\mu$ is a deterministic function. It is shown in \cite{CN} that if the weight distribution satisfies $\P(\tau_e > t) \asymp { e^{\alpha t^{-p}}}$ with $1<p<d$, then for all $\xi>0$,
\begin{equation}\label{maineq2}
\lim_{n\to\infty} \frac{1}{n^p} \log \P\left({\rm T}(0,nx) \geq (\mu(x)+\xi) n \right)= -\alpha 2^{1-p} \lambda_{d,p} \,\xi^,
\end{equation}
where 
\[\lambda_{d,p} = \lim_{n\to\infty} \mathrm{Cap}_{d,p}(n)>0.\]
An interesting behavior appears in the critical case $p=d$. In the same paper, it is proven that for all $\xi>0$,
\begin{equation} \label{eq:resultCN}
\begin{aligned}
\lim_{n\to\infty} \frac{1}{n^d\,\mathrm{Cap_{d,d}}(\ell_n)} \log \P\big({\rm T}(0,nx)\geq (\mu+\xi) n \big)
= -\alpha 2^{1-d} \, \xi^d.
\end{aligned}
\end{equation}
with $\ell_n\sim n/\log n$. Thus, if we combine \eqref{eq:resultCN} with Theorem \ref{th:mainTh}, we obtain that
\begin{equation}
\lim_{n\to\infty} \frac{(\log{n})^{d-1}}{n^d} \log \P\left({\rm T}(0,nx)\geq (\mu+\xi) n \right)
= -\alpha 2^{1-d} \, \xi^d\,{{\bf \rm vol}_{d-1}(S^{d-1}_q)},
\end{equation}
and we see that a logarithmic { correction} appears in the order of magnitude of the large deviations (compare with \eqref{maineq2}).

When $p=2$, the capacity admits a representation in terms of probabilities of hitting time of random walks, see \cite[pp.\ 25\&36]{LP16}.  Moreover, the properties of the capacity have been extensively studied in this case -- in particular, the asymptotics given in Theorem \ref{th:mainTh} are well-known for $p=2$, see e.g.\ \cite[Proposition 6.3.2]{LL10}. The standard way to deal with these estimates for $p=2$ is to rely on Fourier analysis, but this does not adapt well when $p$ is {different from} $2$; thus we follow  another route { here}.

\subsection{Idea of the proof and organization of the paper} 
The main idea is to compare the discrete $p$-capacity in \eqref{general lambda} to its continuous analogue\footnote{Note that there are other way to define the continuous $p$-capacity; for example the standard way is to choose the operator in the energy as $(\sum_{i=1}^d |\partial_i f|^2)^{p/2}$. Our choice of operator is adapted to the form of the discrete capacity in \eqref{general lambda}.}
\begin{equation*}\label{cont_cap}
 \begin{aligned}
 & {\rm Cap}_{d,p}^{c}(A;B) = \\
   &\inf\left\{\int_{\R^d}\sum_{i=1}^d|\partial_i f|^p\ \middle|~\,f\in W^{1,p}(\R^d),f(x)= 1~\forall x\in B,\,f(x)= 0~\forall x\in A\right\},
   \end{aligned}
\end{equation*}
for $A,B\subset\mathbb R^d$.
Under some general conditions on $A$ and $B$, one can for example show by using the tools of $\Gamma$-convergence that
\[\lim_{n\to\infty}n^{p-d}{\rm Cap}_{\Z^d,p}(nA\cap\Z^d;nB\cap\Z^d)={\rm Cap}_{d,p}^{c}(A;B).\]
{In fact, it is possible to deduce this from general results in \cite{AC04}.}
But these tools are not sharp enough to obtain the precise asymptotics in \eqref{eq:asymptoticCap}, so one needs to be able to compare the continuous capacity to the discrete one with a higher degree of precision.

Our strategy to prove Theorem \ref{th:mainTh} goes as follows. We begin by considering the continuous analogue of the variational problem that defines ${\rm Cap}_{d,d}(n)$. The corresponding continuous capacity is easily computed thanks to a symmetry argument that reduces the question to a 1-dimensional problem. Moreover, we can give an explicit expression to the continuous minimizer. Then, we can obtain an upper bound on ${\rm Cap}_{d,d}(n)$ by choosing discrete approximation of the continuous optimizer. In order to obtain a lower bound, we rely on Thomson's variational formulation of the $p$-capacity (Theorem \ref{th:Thomson}), whose proof follows from the electrical network considerations that we present in Section \ref{sec:electrical_network}. The problem is then reduced to constructing a discrete flow that is a good approximation to the optimal continuous flow. To do so, we average locally the optimal continuous flow in a careful way to keep the flow property on the lattice. The proof of Theorem \ref{th:mainTh} is presented in Section \ref{sec:mainProof}.


\section{\texorpdfstring{$p$-electrical network and Thomson principle}{p-electrical network and Thomson principle}}
\label{sec:electrical_network}
The $p$-capacity on a graph can be represented as the effective capacity of an associated $p$-electrical network. If $p\neq 2$, the network has the property of satisfying a \emph{nonlinear} Ohm law, in the sense that if $h$ denotes the potential of the network and $I$ its current, one has
\[I=c\,U(\Delta h),\]
where $c$ is the conductance of the network and $U(u) = |u|^{p-1}\times \mathrm{sign}(u) = |u|^{p-2} u$ is nonlinear when $p\neq 2$. 
The study of nonlinear networks was initiated in \cite{D47,M60}, and there has been later on much focus on the case of infinite graphs, see e.g.\ \cite{HS97} or \cite[Appendix 2]{S94} and references therein.

In this section, we explain the relation between the $p$-capacity and its associated $p$-electrical network, and use it to show Thomson's principle (Theorem \ref{th:ThomsonForSets}). It is also a good opportunity to introduce the key notions of $p$-harmonic functions, flows and strength of flows that will play a central { role} in the proof of Theorem \ref{th:mainTh}.

Although different formulations of Thomson's principle have been shown to hold for general types of functions $U$ and infinite graphs \cite{K16,MS90}, we restrict ourselves to the case of finite graphs and for the function $U(u) = |u|^{p-1}\times \mathrm{sign}(u) = |u|^{p-2} u$. Our motivation is to present a simple approach, inspired from \cite[Section 8]{P99} (which deals with the linear case $p=2$), that {avoids the technicalities arising when considering a more general setting.}

\subsection{\texorpdfstring{$p$-harmonic functions and flows}{p-harmonic functions and flows}}
Let $V$ be { a} finite set of vertices. For $x,y\in V$, we let $\langle x,y\rangle=\{x,y\}$ denote the \emph{undirected edge} with $x,y$ at both ends. We consider a finite graph $G=(V,E)$  with $E$ being a collection of {undirected} edges, where edges may be repeated (we allow multiple edges). In the following, we will also need to consider \emph{directed edges} $\vec{xy}$ that are given by an ordered couples $(x,y)\in V^2$.

We assume the graph $G$ to be \emph{connected}, in the sense that for all $x,y\in V$, there exists a path $x=x_0,x_1,\dots,x_n=y$ such that $\{x_{i-1},x_{i}\}\in E$ for all $i\leq n$.

An electric network on $G$ is defined by the choice of some $a,b\in V$ which are called the \emph{source} and the \emph{sink} of the network, and by a \emph{conductance} $c_{xy}=c_{yx} \in \mathbb R_+$ defined on all edges $\langle x,y\rangle \in E$. We also define $r_{xy} = c_{xy}^{-1}$ to be the \emph{resistance} of $\langle x,y\rangle$. 

Throughout the section, we fix $p>1$ if not stated otherwise. 
\begin{defin}
A function $h\colon V\to \mathbb R$ is said to be $p$-\textbf{harmonic at a vertex} $x\in V$ if
\begin{equation} \label{eq:harmonic}
\sum_{y\sim x} c_{xy} \left| h(y)-h(x)\right|^{p-2} \left(h(y)-h(x)\right) = 0. 
\end{equation}
We say that $h$ is $p$-\textbf{harmonic} (or a \textbf{potential}) if it is $p$-harmonic for all $x\in V\setminus \{a,b\}$.
\end{defin}
An immediate property of $p$-harmonic functions is the maximum principle:
\begin{thm}[Maximum Principle]
If $h$ is $p$-harmonic, then
\[\max_{x\in V} h(x) = \max_{x\in\{a,b\}} h(x).\]
\end{thm}
\begin{proof}
By \eqref{eq:harmonic}, if there exists $x\notin\{a,b\}$ that is a global maximum then for all $y\sim x$ we have $h(x)=h(y)$. Since $G$ is connected, we find by repeating this observation that $h$ is constant. Therefore the maximum is always attained in $\{a,b\}$. 
\end{proof}
\begin{remark} \label{rk:pharmonicEverywhere}
The last argument also entails that if $h$ is $p$-harmonic at every $x\in V$, then $h$ must be constant since $V$ is finite.
\end{remark}
The maximum principle implies in particular that the only $p$-harmonic function satisfying $h(a)=h(b)=0$ is the constant function $h=0$. In fact, one can show that there is an unique $p$-harmonic function up to affine transformations. In order to see that, one can rely on the following fundamental comparison principle of $p$-harmonic functions:

\begin{thm}[Comparison principle. Theorem 3.4 in \cite{HS97}] If $h$ and $g$ are two $p$-harmonic functions such that $h(a)\leq g(a)$ and $h(b)\leq g(b)$, then $h(x)\leq g(x)$ for all $x\in V$.
\end{thm}

\begin{cor}\label{linear transformation}
Let $h$ be a $p$-harmonic function such that $h(a)=0$ and $h(b)=1$. Then, for every $p$-harmonic function $g$, there exist $\alpha,\beta\in \mathbb R$ such that $g = \alpha h+\beta $.
\end{cor}
\begin{proof}
Let $\beta = g(a)$ and $\alpha = g(b)-g(a)$. Then, the function $\tilde g := \alpha h + \beta$ is also $p$-harmonic and satisfies $g (a) = \tilde g(a)$ and $ g (b) = \tilde g(b)$. By the comparison principle, it must be that $g = \tilde g$.
\end{proof}

\begin{defin} \label{def:flow}
A \emph{flow from $a$ to $b$} -- or simply \emph{flow} -- is a function $\theta(\vec{xy})$ on directed edges that is antisymmetric, in the sense that for all $x,y\in V$, $\theta(\vec{xy}) = - \theta(\vec{yx})$, and which satisfies the \emph{node law}:
\begin{equation} \label{eq:node_law}
\forall x\in G\setminus \{a,b\}, \quad \sum_{y\sim x} \theta(\vec{xy}) = 0.
\end{equation}
\end{defin}

If $h$ is a $p$-harmonic function, one can define a flow $I$ that we call the \emph{current flow} associated to $h$ by the relation:
\begin{equation} \label{eq:nonlinearOhmLaw}
I(\vec{xy}) = c_{xy} \left|h(y)-h(x)\right|^{p-2} \left(h(y)-h(x)\right).
\end{equation}

We say that a flow $\theta$ satisfies the \emph{cycle law} if for all $x_0,\dots,x_n\in V$ such that $x_n=x_0$, and $\vec{e_i} = \overrightarrow{x_{i-1} x_i}$,
\begin{equation} \label{eq:cycle_law}
\sum_{i=1}^n r_{ {e_i}}^{\frac 1 {p-1}} \frac{\theta\left(\vec {e_i}\right)}{\left| \theta\left(\vec {e_i}\right)\right|^{\frac{p-2}{p-1}}} = 0.
\end{equation}
Note that we set the summand in the above sum to be $0$ whenever $\theta\left(\vec {e_i}\right) = 0$.
By definition, any current flow $I$ satisfies the cycle law. Conversely, any flow that satisfies the cycle law is a current flow:

\begin{thm} \label{th:equivalenceHarmonicFlow}
Let $\theta$ be a flow satisfying the cycle law \eqref{eq:cycle_law}. Then, there exists a $p$-harmonic function $h$ such that 
\begin{equation} \label{eq:flow_to_pharmonic}
\theta(\vec{xy}) = c_{xy} \left|h(y)-h(x)\right|^{p-2} \left(h(y)-h(x)\right).
\end{equation}
Moreover, $p$-harmonic functions satisfying \eqref{eq:flow_to_pharmonic} are unique up to an additive constant.
\end{thm}
\begin{proof}
For any $x\in V$, let $a=x_0,x_1,\dots,x_k=x$ be any path from $a$ to $x$ and define
\begin{equation} \label{eq:primitive_h}
h(x) = h(a) + \sum_{i=1}^k r_{ {e_i}}^{\frac 1 {p-1}} \frac{\theta\left(\vec {e_i}\right)}{\left| \theta\left(\vec {e_i}\right)\right|^{\frac{p-2}{p-1}}}\,.
\end{equation}
Note that by the cycle law, $h$ is well defined.
For all $x\sim y$, we have
\begin{equation} \label{eq:diff_h}
h(y)-h(x) = r_{xy}^{\frac 1 {p-1}} \frac{\theta\left(\vec {xy}\right)}{\left| \theta\left(\vec {xy}\right)\right|^{\frac{p-2}{p-1}}}\,,
\end{equation}
so that \eqref{eq:flow_to_pharmonic} is verified  and, since $c_{xy} r_{xy}=1$ when $c_{xy}\neq 0$,
\begin{align*}
& \sum_{y\sim x} c_{xy} \left| h(y)-h(x)\right|^{p-2} \left(h(y)-h(x)\right) \\
& = \sum_{y\sim x} c_{xy} r_{xy}^{\frac 1 {p-1}} r_{xy}^{\frac {p-2} {p-1}} \frac{|\theta\left(\vec {xy}\right)|^{p-2}\theta\left(\vec {xy}\right)}{\left| \theta\left(\vec {xy}\right)\right|^{\frac{p-2}{p-1}(p-2)}\left| \theta\left(\vec {xy}\right)\right|^{\frac{p-2}{p-1}}}\\
& = \sum_{y\sim x} \theta\left(\vec {xy}\right) =0,
\end{align*}
which proves the first part of the theorem. Suppose now that another $p$-harmonic function $h_1$ satisfies \eqref{eq:flow_to_pharmonic}. The function $h_1$ therefore  satisfies \eqref{eq:diff_h} and in turn \eqref{eq:primitive_h}, which entails that $h_1=h+h_1(a)-h(a)$.
\end{proof}

Define the \emph{strength} of a flow by $\Vert \theta \Vert := \sum_{x\sim a} \theta(\vec{ax})$. Note that since $\theta$ is antisymmetric, one has $\sum_{y\in V}\sum_{x\sim y} \theta(\vec{xy}) = 0$, so by the node law we find that $\Vert \theta \Vert = \sum_{x\sim b} \theta(\vec{xb})$. 
\begin{thm} \label{th:strengthEquality}
Let $\theta$ and $\eta$ be two flows that satisfy the cycle law. If $\Vert \theta \Vert = \Vert \eta \Vert$, then $\theta = \eta$.
\end{thm}
\begin{proof}
Let $h$ be a $p$-harmonic function satisfying \eqref{eq:flow_to_pharmonic}, and let $g$ be a $p$-harmonic function associated to $\eta$ via the same relation. If $h$ is constant, then $\Vert \theta \Vert = 0$ so that $\Vert \eta \Vert = 0$. Therefore $g$ is $p$-harmonic at every $x\in V$, and by Remark \ref{rk:pharmonicEverywhere} one finds that $g$ is constant and thus $\theta = \eta = 0$.

If $h$ is non-constant, then by the uniqueness property there exist $\alpha,\beta\in \mathbb R$ such that $g=\alpha h + \beta$. Hence $\eta = \alpha \theta$, and since $\Vert \theta \Vert\neq 0$ (otherwise $h$ would be constant), we obtain from $\Vert \eta \Vert = \Vert \theta \Vert$ that $\eta = \theta$.
\end{proof}

\subsection{The effective resistance}
\begin{thm} \label{th:defEffectiveResistance}
Let $h$ be a non-constant $p$-harmonic function and $I$ be its current flow. The ratio:
\begin{equation} \label{eq:effect_res_def}
R_{p}(a,b):= \frac{|h(b)-h(a)|^{p-2}(h(b)-h(a))}{\Vert I \Vert},
\end{equation}
is well defined, non-zero and does not depend on the choice of $h$.
\end{thm}
\begin{proof}
First observe that if $h$ is a non-constant $p$-harmonic function, then $\Vert I \Vert$ must be non-zero, otherwise $h$ would be $p$-harmonic at every point $x\in V$ and by Remark \ref{rk:pharmonicEverywhere}, $h$ would be constant. Moreover, $h(a)$ has to be different from $h(b)$ by the maximum principle. 

Now let  $h_1,h_2$ be two non-constant $p$-harmonic functions  and $I_1,I_2$ be their associated current flows. Up to exchanging $h_1$ with $-h_1$ or $h_2$ with $-h_2$, one can assume that both $\Vert I_1\Vert$ and $\Vert I_2\Vert$ are positive. Denote by $\bar h_1 = h_1/\Vert I_1 \Vert^{\frac{1}{p-1}}$ and $\bar h_2 = h_2/\Vert I_2 \Vert^{\frac{1}{p-1}}$. Further let $\theta_1,\theta_2$ be the current flows associated to $\bar h_1,\bar h_2$. We have 
\[\Vert \theta_1 \Vert = \Vert I_1\Vert^{-1} \sum_{x\sim a} c_{ax} \left|h_1(x)-h_1(a)\right|^{p-2} \left(h_1(x)- h_1(a)\right) = 1, \]
and the same holds for $\theta_2$, so that by Theorem \ref{th:strengthEquality} we have $\theta_1 = \theta_2$. Hence by Theorem \ref{th:equivalenceHarmonicFlow}, $\bar h_1$ and $\bar h_2$ only differ from an additive constant, so in particular the ratios \eqref{eq:effect_res_def} with $h_1$ and $h_2$ are equal.
\end{proof}

We call $R_{p}(a,b)$ the $p$-\emph{effective resistance} and its inverse $C_{p}(a,b):= R_{p}(a,b)^{-1}$ the $p$-\emph{effective conductance} or $p$-\emph{capacity}. The following Dirichlet and Thomson principles offer effective ways of estimating the $p$-effective resistance and the $p$-capacity.

\begin{thm}[Dirichlet's principle] \label{th:Dirichlet}
One has: 
\[C_{p}(a,b) = \inf_{h\colon V\to \mathbb R}\left\{\frac{1}{2}\sum_{x\in V} \sum_{y\sim x} c_{xy}|h(y)-h(x)|^p, h(a)=0,h(b)=1\right \}.\]
Moreover, the infimum is attained at a $p$-harmonic function $h$.
\end{thm}
\begin{remark}
This entails that $C_p(a,b)= {\rm Cap}_{G,p}(\{a\},\{b\})$ when $c_{xy} \equiv 1$.
\end{remark}
\begin{remark}
The theorem shows in particular that non-constant $p$-harmonic functions always exist and that the $p$-capacity is always positive.
\end{remark}
\begin{proof}
The fact that it is $p$-harmonic is straightforward by differentiating at all points of $V\backslash\{a,b\}$. Now let $h$ be the $p$-harmonic function minimizing the above infimum. We have
\begin{align*}
&\frac 1 2 \sum_{x\in V} \sum_{y\sim x} c_{xy}|h(y)-h(x)|^p \\
& = \frac 1 2 \sum_{x\in V} \sum_{y\sim x} c_{xy} (h(y)-h(x)) |h(y)-h(x)|^{p-2}(h(y)-h(x))\\
& = \frac 1 2 \sum_{x\in V} \sum_{y\sim x} c_{xy} h(y) |h(y)-h(x)|^{p-2}(h(y)-h(x))\\
& \qquad - \frac 1 2 \sum_{x\in V} \sum_{y\sim x} c_{xy} h(x) |h(y)-h(x)|^{p-2}(h(y)-h(x))\\
& = -\sum_{x\in V} \sum_{y\sim x} c_{xy} h(x) |h(y)-h(x)|^{p-2}(h(y)-h(x)).
\end{align*}
By $p$-harmonicity of $h$, the last sum further equals $-h(a)\Vert I \Vert - h(b)(-\Vert I \Vert)$ where $I$ is the current flow associated to $h$. Since $h(b)=1$ and $h(a)=0$, this equals $C_{p}(a,b)$.
\end{proof}

We say that a flow $\theta$ is a \emph{unit flow} if $\Vert \theta \Vert = 1$. Dirichlet's principle can be restated as a problem involving flows as follows:
\begin{thm}[Thomson's principle] \label{th:Thomson}
When $p>1$, the following holds:
\[R_{p}(a,b) = \inf \left\{\sum_{e\in E} r_e^{\frac{1}{p-1}}| \theta(e)|^{\frac{p}{p-1}}, \,\theta \text{ is a unit flow}\right\}^{p-1}.\]
Moreover, the infimum is attained at a current flow.
\end{thm}
\begin{proof}
We first show that any minimizer $\theta$ of the above problem must satisfy the cycle law. Let $\vec e_1,\dots,\vec e_n$ form a cycle and set $\gamma(\vec e_i)= 1$ for all $1\leq i\leq n$, and $\gamma(e)=0$ on all other edges. Note that $\gamma$ defines a flow. Denoting by $\mathcal E(\theta)$ the energy $\sum_{e\in E} r_e^{\frac{1}{p-1}}| \theta(e)|^{\frac{p}{p-1}}$, for all $\varepsilon\in \mathbb R$ we have
\begin{align*}0\leq \mathcal E(\theta + \varepsilon \gamma) - \mathcal E(\theta) & = {\sum_{i=1}^n} r_e^{\frac{1}{p-1}}| (\theta + \varepsilon)(\vec e_i)|^{\frac{p}{p-1}} - \sum_{i=1}^n r_{e_i}^{\frac{1}{p-1}}| \theta(\vec e_i)|^{\frac{p}{p-1}}\\
& =\frac{p}{p-1}\varepsilon \sum_{i=1}^n r_{e_i}^{\frac{1}{p-1}} |\theta(\vec e_i)|^{\frac{p}{p-1}}  + |\varepsilon|^{\frac{p}{p-1}} \sum_{i=1}^n r_{e_i}^{\frac{1}{p-1}}  \mathbf{1}_{\theta(\vec e_i) = 0} + O(\varepsilon^2).
\end{align*}
 By letting $\varepsilon \to 0$ from below and from above, we see that necessarily 
\[\sum_{i=1}^n r_{e_i}^{\frac{1}{p-1}} |\theta(\vec e_i)|^{\frac{p}{p-1}} =0,\]
therefore $\theta$ is a current flow. 

We will now prove that $\mathcal E (\theta)^{p-1} = R_{p}(a,b)$ for the minimizer $\theta$. Let $h$ be the $p$-harmonic function associated to $\theta$ such that $h(a)=0$. We have 
\begin{align*}
\sum_{e\in E} r_e^{\frac 1 {p-1}} |\theta(e)|^{\frac{p}{p-1}} = \frac{1}{2}\sum_{x\in V}\sum_{x\sim y} c_{xy} |h(y)-h(x)|^p. 
\end{align*}
In the last sum, if we replace $h$ by $h/h(b)$ which satisfies $h(b)=1$ and $h(a)=0$, we see from Dirichlet's principle that the sum equals $|h(b)|^{p} C_{p}(a,b)$. Since $\Vert \theta \Vert = 1$, we have $|h(b)|^{p-1} = R_{p}(a,b)$.  Therefore,
\[\left(|h(b)|^{p} C_{p}(a,b)\right)^{p-1}=\left(|h(b)| R_{p}(a,b) C_{p}(a,b)\right)^{p-1}=|h(b)|^{p-1}=R_{p}(a,b).\]

\end{proof}
Let ${\rm Cap}_{G}(A;B)$ denote the capacity between two sets of vertices defined as in \eqref{general lambda}, except that we allow here the presence of a resistance term $r_e$ in the definition of the energy $\sum_{e=\langle x,y\rangle \in E} r_{xy}|f(x)-f(y)|^p$. Thomson's principle can be easily extended to the case where the source and the sink are sets of vertices:
\begin{thm}[Thomson's principle. Set formulation] \label{th:ThomsonForSets}
\begin{equation} \label{eq:thomson_AB}
{\rm Cap}_{G,p}(A;B)^{-1}=\inf \left\{\sum_{e\in E} r_e^{\frac{1}{p-1}}| \theta(e)|^{\frac{p}{p-1}},\, \theta \text{ is a unit flow from $A$ to $B$}\right\}^{p-1},
\end{equation}
where a flow from $A$ to $B$ is defined as in Definition \ref{def:flow} with $a,b$ replaced by $A,B$. Moreover, the infimum is attained at a current flow.
\end{thm}
\begin{proof}
Let $A,B\subset G$ such that $A\cap B=\emptyset$. Consider the reduced network where all points $x\in A$ are merged to a single vertex $a$ and all the edges that ended in $A$ now end at $a$ with the same conductances. Do the same for $B$ and merge them to a single vertex $b$. Choose the potential on $A$ and on $B$ to be constant. Theorem \ref{th:Thomson} then implies \eqref{eq:thomson_AB}.
\end{proof}

Of independent interest, we observe that Thomson's principle still holds for the case $p=1$. It can be stated as follows:
\begin{cor}
When $p=1$, we have:
\begin{equation}
{\rm Cap}_{G,1}(A;B)^{-1}=\inf \left\{|r\theta|_\infty,\, \theta \text{ is a unit flow from $A$ to $B$}\right\}.
\end{equation}
\end{cor}
\begin{proof}
First note that the additional restriction $f(x)\in[0,1]$ in \eqref{general lambda} does not change the value of the capacity. Moreover, on this restriction, $\sum r_{xy} |f(x)-f(y)|^{p}$ is decreasing in $p$. Hence,
\[{\rm Cap}_{G,1}(A;B)\geq {\rm Cap}^{1+\e}(A;B).\]
We take   a sequence $\e_k\to 0$ so that 
\[{\rm Cap}_{G,1}(A;B) \geq \liminf_{\e\to 0} {\rm Cap}_{G,1+\e_k}(A;B) =\lim_{k\to\infty} {\rm Cap}_{G,1+\e_k}(A;B).\] Let $h_{\e}$ be a harmonic function for ${\rm Cap}_{G,1+\e}(A;B)$ (i.e.\ $h_\e$ is harmonic on $V\setminus (A\cup B))$. Then there exists a subsequence $\e_k'$ of $\e_k$ and a function $h\colon V\to\R$ such that $h_{\e_k'}(x)\to h^*(x)$ for all $x\in V$. This yields
\al{
  {\rm Cap}_{G,1}(A;B)&\leq \sum_{\langle x, y\rangle \in E} r_{xy} |h^*(x)-h^*(y)|\\
  &= \lim_{k\to\infty} \sum_{\langle x, y\rangle \in E} r_{xy} |h_{\e_k'}(x)-h_{\e_k'}(y)|\\
  &=\liminf_{\e\to 0} {\rm Cap}_{G,1+\e}(A;B).
}
Hence, $\lim_{\e\to 0}  {\rm Cap}_{G,1+\e}(A;B)= {\rm Cap}_{G,1}(A;B).$\\

Thomson's principle states that  ${\rm Cap}_{G,1+\e}(A;B)= | r^{\frac{1}{1+\e}}\theta_\e|_{\frac{1+\e}{\e}}^{-(1+\e)}$ for a current flow $\theta_\e$. We can take a subsequence $\e_k$ such that  $\theta_{\e_k}$ converges to a unit flow $\theta^*$. We have:
\[ \lim_{k\to \infty} \left|r^{\frac{1}{1+\e_k}} \theta_{\e_k}\right|_{\frac{1+\e_k}{\e_k}}^{-(1+\e_k)}=|r\theta^*|_{\infty}^{-1}\leq  \sup  \left\{|r\theta|^{-1}_{\infty}: \theta \text{ is a unit flow}\right\}.\]
On the other hand, if we let $\theta_*$ maximize $|r\theta|^{-1}_{\infty}$, we obtain from the fact that $\theta_{\e_k}$ is a maximizer for the $(1+\e_k)$-capacity
\[ \lim_{k\to \infty} \left|r^{\frac{1}{1+\e_k}} \theta_{\e_k}\right|_{\frac{1+\e_k}{\e_k}}^{-(1+\e_k)}\geq \lim_{k\to \infty} \left|r^{\frac{1}{1+\e_k}} \theta_{\star}\right|_{\frac{1+\e_k}{\e_k}}^{-(1+\e_k)}=  |r\theta_*|^{-1}_{\infty}.\]
Putting things together,
\al{
 \sup  \left\{|r\theta|^{-1}_{\infty}: \theta \text{ is a unit flow}\right\}&=\lim_{k\to \infty} \left|r^{\frac{1}{1+\e_k}} \theta_{\e_k}\right|_{\frac{1+\e_k}{\e_k}}^{-(1+\e_k)}\\
 &=  \lim_{\e\to 0}  {\rm Cap}_{G,1+\e}(A;B)= {\rm Cap}_{G,1}(A;B).
}
\end{proof}

\section{Proof of Theorem \ref{th:mainTh}}
\label{sec:mainProof}

\subsection{The continuous analogue}
As mentioned in the introduction, it turns out that the analogous continuous $p$-capacity problem can be easily solved. Given $q\geq 1$, we define 
\[B^d_{q}=\left\{x\in \R^d:~ |x|_q\leq 1\right\},\,S_q^{d-1}=\left\{x\in \R^d:~ |x|_q=1\right\}.\]
In all that follows, $q$ will be chosen as
\begin{equation}
q=\frac{d}{d-1}.
\end{equation} Let,
\[{\rm Cap}^c_d(n)=\inf\left\{ \int_{n B^d_{q}} \sum_{i=1}^d \left|\partial_i f(x)|^d \dd x \,:~ f(x)=0 \ \forall x\in B^d_q,\,f(y)=1\ \forall y\in nS^{d-1}_q \right.\right\}.\]
\begin{thm} \label{th:cont_cap} We have
\[{\rm Cap}^c_d(n)=\frac{{\bf \rm vol}_{d-1}(S^{d-1}_q)}{(\log n)^{d-1}},\]
where ${\bf \rm vol}_{d-1}$ is the $d-1$-dimensional volume.
\end{thm}
\begin{proof}
  The proof follows from a standard computation.
If we take 
\[f(x)=\frac{\log{|x|_{q}}} {\log{n}}=\frac{\log{ \sum_{i=1}^d |x_i|^q}}{q \log{n}},\]
 then 
\[|\partial _i f(x)|= \frac{|x_i|^{q-1}}{\log{n} \sum_i |x_i|^q}.\]
Hence, 
\aln{
\int_{n B^d_{q}} \sum_{i=1}^d |\partial_i f(x)|^d \dd x &=\frac{1}{(\log{n})^d} \int_{n B^d_{q}} \sum_{i=1}^d \left(\frac{|x_i|^{q-1}}{\sum_i |x_i|^q}\right)^d \dd x\notag\\
&= \frac{1}{(\log{n})^d} \int_{n B^d_{q}}  \left( \sum_i |x_i|^q \right)^{-(d-1)} \dd x.\label{added}
}
Using polar coordinates, 
\ben{
\text{$x=r\theta$ with $r=|x|_q>0$ and $\theta=x/|x|_q\in S_q^{d-1}$},\label{polar coordinate}
}
 so that $\dd x=r^{d-1} \dd r \dd \theta$, since $|x|_q^{-(d-1)q}\,|x|_q^{d-1}=r^{-1}$, \eqref{added} is equal to
\aln{
& \frac{{\bf \rm vol}_{d-1} (S^{d-1}_q)}{(\log{n})^d} \int_{1}^n  r^{-1} \dd r=\frac{{\bf \rm vol}_{d-1}(S^{d-1}_q)}{(\log n)^{d-1}}.\label{continuous computation}
}
This gives an upper bound on the continuous capacity.

Next, we consider the lower bound. We use polar coordinates as in \eqref{polar coordinate} again. We denote $\theta=(\theta_i)_{i=1}^d$.
By the H\"older inequality, we have:
\al{
|\partial_r f(r\theta)|^d&= \left|\sum_i \theta_i \partial_i f(x)\right|^d\\
&\leq  \left(\sum_i |\theta_i|^{q}\right)^{1/q} \left(\sum_i |\partial_i f(x)|^d\right)\\
&=\sum_i |\partial_i f(x)|^d.
}
Hence, 
\al{
& \int_{n B^d_{q}} \sum_{i=1}^d |\partial_i f(x)|^d \dd x\\
 &\geq \iint_{[1,n]\times S^{d-1}_q}  |\partial_r f(r \theta)|^{d} r^{d-1}  \dd r \dd \theta \\
&\geq {\bf \rm vol}_{d-1}(S^{d-1}_q)\inf\left\{ \int_{1}^n  |g'(t)|^{d} t^{d-1}  \dd t:~g:[1,n]\to \R,\,g(1)=0,\,g(n)=1)\right\}\\ 
&=\frac{{\bf \rm vol}_{d-1}(S^{d-1}_q)}{(\log n)^{d-1}},
}
where we have changed { variables} in the first line, and used that the variational problem of dimension $1$ in the last line is exactly solvable ($g(t)=\frac{\log t}{\log n}$ attains the infimum).
\end{proof}

\subsection{Upper bound} 
\begin{lem}
\[\limsup_{n\to\infty} (\log n)^{d-1} \, {\rm Cap}_{d,d}(n)\leq {\bf \rm vol}_{d-1}(S^{d-1}_q),\]
\end{lem}
\begin{proof}
Let us set $f(x)=\frac{\log{(d|x|_q+1)}}{\log{n}}$. Then for $x\in\partial D_n$, {we have} $f(x)\geq 1$, and also $f(0)=0$. Moreover, for $x\in D_n$, 
\aln{
|f(x+\mathbf{e}_i)-f(x)|&= (1+o(1))\frac{|x_i|^{q-1}}{\log{n} \sum_i |x_i|^q}.\label{approx pe}
}
Hence, we have
\aln{
\sum_{x,y\in E} |f(x)-f(y)|^d&\leq (1+o(1)) (\log{n})^{-d} \sum_{x\in D_n}  \left(\frac{|x_i|^{q-1}}{\sum_i |x_i|^q}\right)^d\notag\\
&\leq (1+2o(1)) (\log{n})^{-d} \int_{n B^d_{q}} \sum_{i=1}^d \left(\frac{|x_i|^{q-1}}{\sum_i |x_i|^q}\right)^d \dd x\notag\\
&\leq (1+2o(1))\frac{{\bf \rm vol}_{d-1}(S^{d-1}_q)}{(\log{n})^{d-1}},\label{discrete upper bound}
}
which gives an upper bound.
\end{proof}

Before turning to the proof of the lower bound, let us introduce the notion of $p$-harmonic functions on $\mathbb R^d$.

\subsection{\texorpdfstring{A $p$-harmonic function and its associated flow}{A p-harmonic function and its associated flow}} \label{subsec:pharmo}
We say that a continuous function $u$ is $\Delta_p$-\emph{harmonic}\footnote{Note that the standard definition of a $p$-harmonic function would be with the operator $\diverg(|\nabla u|^{p-2} \nabla u)$. This is why we call it $\Delta_p$-harmonic instead.} on $A\subset \mathbb R^d$ if for all $x\in A$ one has $\Delta_p u(x)=0$, where
\[\Delta_p u = \diverg\left(\left(\left|\frac{\partial u}{\partial x_{i}}\right|^{p-2}\frac{\partial u}{\partial x_{i}} \right)_{i=1,\dots,d}\right).\]

 We observe that when $p=d$,
 the function
 \ben{\label{function g def}
   g(x)=q^{-1}\log \Vert x \Vert_{q}^q \text{ where } q=\frac{d}{d-1},
 }
 is $\Delta_p$-harmonic in $\mathbb R^d\setminus \{0\}$. Indeed, one has
\[
\frac{\partial g}{\partial x_{i}}(x)=\frac{|x_i|^{\frac{d}{d-1}-1}}{|x|_q^q}\frac{x_i}{|x_i|},
\]
so that
\[
\Theta_i(x) := \left|\frac{\partial g}{\partial x_{i}}(x)\right|^{d-2}\frac{\partial g}{\partial x_{i}}(x) =  \frac{x_i}{|x|_q^{d}}.
\]
Therefore,
\[\frac{\partial \Theta_i}{\partial x_{i}}(x) =  \frac{(\sum_{i=1}^d |x_i|^{q})^{d-2}}{|x|_q^{2d}} \left( \sum_{i=1}^d |x_i|^{q} - d |x_i|^q \right),
\]
from where we see that
\begin{equation}
\forall x\in \mathbb R^d\setminus \{0\}, \quad \mathrm{div}\, \Theta(x) = 0,
\end{equation}
with $\Theta=(\Theta_i)_{i=1,\dots,d}$,
i.e.\ $g$ is $\Delta_p$-harmonic. In particular, the vector field $\Theta$ defines a continuous \emph{flow} (i.e.\ a divergence-free vector field) on $\mathbb R^d\setminus \{0\}$ associated to the $\Delta_p$-harmonic function $g$.

\subsection{Lower bound via discretization of flow}
\begin{lem} We have,
  \[\limsup_{n\to\infty} (\log n)^{d-1} \, {\rm Cap}_{d,d}(n)\geq {\bf \rm vol}_{d-1}(S^{d-1}_q).\]
\end{lem}
\begin{proof}
We will use Thomson's principle (Theorem \ref{th:Thomson}). The divergence-free function $\Theta$ defined in Section \ref{subsec:pharmo} is a good candidate for a flow, but one needs first to construct an approximating discrete flow out of it. To do so, we use a construction described by Lyons in \cite[Section 6]{L83}. Let $x,y$ be two neighboring points in $\mathbb Z^d$ and denote by $S_{xy}$ the face of the cube of length one centered at $x$ that is perpendicular to $\vec{xy}$.  We then define the flow $\theta(\vec{xy})$ to be the value of the flow of $\Theta$ that goes through $S_{xy}$ in the direction from $x$ to $y$. By the divergence theorem, $\theta$ defines a flow on $\mathbb Z^d \setminus \{0\}$.

Now let $x\in\mathbb Z^d$ and $y=(x_1,\dots,x_i \pm 1,\dots,x_d)$ for some $i\leq d$. The discrete flow $\theta$ satisfies
\[\theta(\vec{xy}) = \pm  \int_{\prod_{j\neq i} [x_j-\frac{1}{2},x_j+ \frac{1}{2}]} \frac{x_i \pm \frac{1}{2}}{\left(\sum_{j\neq i} |z_j|^q + |x_i\pm \frac{1}{2}|^q\right)^{d-1}} \prod_{j\neq i} \dd z_j.\]
By the reverse triangular inequality, the above denominator is bounded from below by $(|x|_q - |a|_q)^{d}$ where $a=(1/2,\dots,1/2)$. Since $|a|_q \leq d$, we obtain that  for all $x$ such that $|x|_q> d$, 
\[|\theta(x,y)| \leq  \frac{|x_i\pm \frac{1}{2}|}{(|x|_q-d)^d},\]
and hence, by the triangular inequality,
\[\frac{1}{2}\sum_{y\sim x} |\theta(x,y)|^{\frac{d}{d-1}}\leq   \frac{(|x|_q+|a|_q)^{q}}{(|x|_q-d)^{dq}}\leq   \frac{(|x|_q+d)^{q}}{(|x|_q-d)^{dq}}.\]


 We now fix $\e>0$ and choose $R=R_\varepsilon>0$ big enough such that for all $x\in\Z^d$ with $|x|_q>R$ and any neighboring point $x'\in \prod_{i\leq d} [x_i-1/2,x_i+1/2]$,  we have (recall that $q=\frac{d}{d-1}$), 
 \[\frac{(|x|_q+d)^{q}}{(|x|_q-d)^{dq}} \leq (1+\e)|x'|_q^{-d}.\]
This implies in particular that for all $|x|_q >R$,
\al{
\frac{(|x|_q+d)^{q}}{(|x|_q-d)^{dq}}&=\int_{\prod_{i\leq d} [x_i-1/2,x_i+1/2]} \frac{(|x|_q+d)^{q}}{(|x|_q-d)^{dq}}\dd x'\\
&\leq (1+\e)\int_{\prod_{i\leq d} [x_i-1/2,x_i+1/2]}|x'|_q^{-d}\dd x'.
}
Then we let $c>0$ be such that $c^{-1} |x|_\infty \leq |x|_q \leq c |x|_\infty$, so that
\begin{align}
&\frac{1}{2} \sum_{x\in \Z^d:\, 0\leq |x|_\infty \leq n} \sum_{y\sim x} |\theta(x,y)|^{\frac{d}{d-1}} \nonumber
\\& \leq \frac{1}{2} \sum_{x\in \Z^d:\, 0 \leq |x|_q \leq R} \sum_{y\sim x} |\theta(x,y)|^{\frac{d}{d-1}} + \sum_{x\in \Z^d:\, R < |x|_q \leq cn} \frac{(|x|_q+d)^{q}}{(|x|_q-d)^{dq}}.
\label{sum approx}
\end{align}
The last sum of \eqref{sum approx} is less than 
\al{
(1+\e)&\int_{1\leq |x'|_q \leq  cn+2}|x'|_q^{-d}\dd x'\\
& \leq (1+\e) {\rm vol}_{d-1}(S_q^{d-1})\log(cn+2),
}
by the same computation as in \eqref{continuous computation}, while the first sum of \eqref{sum approx} is bounded from above by $CR^d$
for some $C=C(d)>0$. This can be seen by using that $|\Theta_{i}(x)| \leq \frac{|x|_\infty}{|x|_q^d} \leq  \frac{c}{|x|^{d-1}_q}$, which implies that $|\Theta|_\infty$ is uniformly bounded from above on all faces $S_{xy}$ and thus that $|\theta(x,y)|$ is uniformly bounded { from} above. Therefore,
\begin{align*}
&\limsup_{n\to\infty}\, (\log n)^{-1} \frac{1}{2} \sum_{x\in \Z^d:\, 0\leq |x|_\infty \leq n} \sum_{y\sim x} |\theta(x,y)|^{\frac{d}{d-1}}  \leq  (1+\e) {\rm vol}_{d-1}(S_q^{d-1}).
\end{align*}
Since $\varepsilon$ is arbitrary, we obtain from Thomson's principle (Theorem \ref{th:ThomsonForSets}) that
\begin{equation} \label{eq:applThomson} \liminf_{n\to\infty} (\log n)^{d-1} \mathrm{Cap}_{d,d}(n) \geq  \Vert \theta \Vert^{q(d-1)} \, {\bf \rm vol}_{d-1}(S^{d-1}_q)^{1-d}.
\end{equation}

We are now left with computing $\Vert \theta \Vert$. {We know that ${\rm div}\Theta=0$ on $\R^d\setminus\{0\}$, and so from the divergence theorem we obtain that}
\[\Vert \theta \Vert=\int_{\partial D_{n-1/2}} \Theta \cdot \dd \vec{S} = \int_{nS_q^{d-1}} \Theta \cdot \dd \vec{S},\] 
and from the same theorem applied to $g\Theta$ with $g$ {from \eqref{function g def}}, that also
\[\int_{(n B_q^d) \setminus B_q^d} \sum_{i=1}^d |\partial_i g(x)|^d  +  g \Delta_p g \dd x = \int_{\partial ((n B_q^d) \setminus B_q^d)} (g\Theta)\cdot \dd \vec{S}.\]
Therefore, combining that $g=0$ on $S_q^{d-1}$ and $g=\log n$ on $n S_q^{d-1}$ with the fact that $g$ is $\Delta_p$-harmonic, we find that
\begin{equation} \label{eq:analogueDirichlet}
\int_{(n B_q^d) \setminus B_q^d} \sum_{i=1}^d |\partial_i g(x)|^d = (\log n) \int_{nS_q^{d-1}} \Theta \cdot \dd \vec{S}.
\end{equation}
Furthermore, we know from \eqref{continuous computation}, where $f=g/\log n$, that
\[\int_{(n B_q^d) \setminus B_q^d} \sum_{i=1}^d |\partial_i g(x)|^d \dd x = (\log n)\, {\bf \rm vol}_{d-1}(S^{d-1}_q),\]
and thus
\[\Vert \theta \Vert  = \int_{n S_q^{d-1}}  \Theta \cdot \dd \vec{S} = {\bf \rm vol}_{d-1}(S^{d-1}_q).\]

We thus obtain from \eqref{eq:applThomson} that
\[\liminf_{n\to\infty}\, (\log n)^{d-1} \mathrm{Cap}_{d,d}(n) \geq {\bf \rm vol}_{d-1}(S^{d-1}_q).
\]
\end{proof}
\begin{rem}
Equation \eqref{eq:analogueDirichlet} expresses that the continuous Dirichlet energy  is equal to the strength of the continuous flow associated to the harmonic function $g$ with boundary conditions $g=0$ on $S_q^{d-1}$ and $g=\log n$ on $n S_q^{d-1}$. This is the continuous analogue of what is shown in the proof of Theorem \ref{th:Dirichlet} (Dirichlet's principle). In particular, the integration by part that we used to obtain \eqref{eq:analogueDirichlet} may be compared to the manipulation of sums that is performed in the proof of Theorem \ref{th:Dirichlet}.
\end{rem}

\section*{Declarations}

\subsection*{Acknowledgements}
The authors would like to thank Bobo Hua for helpful discussions. 
\subsection*{Funding}
Shuta Nakajima is supported by SNSF grant 176918. Florian Schweiger is supported by the Foreign Postdoctoral Fellowship Program of the Israel Academy of Sciences and Humanities. Cl\'ement Cosco and Florian Schweiger acknowledge that this project has received funding from the European Research Council (ERC) under the European Union's Horizon 2020 research and innovation programme (grant agreement No. 692452). 

\subsection*{Conflicts of interests/Competing interests} The authors have no relevant financial or non-financial interests to disclose.

\subsection*{Data availability statement} No new data was generated in relation to this man-
uscript.



\begin{thebibliography}{99}

\bibitem{AC04}
R. Alicandro and M. Cicalese. \newblock
A general integral representation result for continuum limits of discrete energies with superlinear growth. \newblock
\emph{SIAM J. Math. Anal.} 36 (2004), no. 1, 1-37



\bibitem{CN}
C. Cosco and S. Nakajima.
\newblock A variational formula for large deviations in First-passage percolation under tail estimates. \newblock
\emph{arXiv:2101.08113} (2021)

\bibitem{D47}
R. J. Duffin. \newblock 
Nonlinear networks. IIa. \newblock 
\emph{Bull. Amer. Math. Soc.} 53 (1947), 963-971.



\bibitem{HS97}
I. Holopainen and P. M. Soardi. \newblock
$p$-harmonic functions on graphs and manifolds.  \newblock
\emph{Manuscripta Math.} 94 (1997), no. 1, 95-110.

\bibitem{K16}
A. Kasue. \newblock
A Thomson's Principle and a Rayleigh's Monotonicity Law for Nonlinear Networks, \newblock
\emph{Potential Anal.} 45 (2016), no. 4, 655-701.


\bibitem{LP16}
R. Lyons and Y. Peres. \newblock
Probability on trees and networks. \newblock
\emph{Cambridge Series in Statistical and Probabilistic Mathematics}, 42. Cambridge University Press, New York, 2016.

\bibitem{LL10}
G. Lawler and V. Limic. \newblock
Random walk: a modern introduction. \newblock
\emph{Cambridge Studies in Advanced Mathematics}, 123. Cambridge University Press, Cambridge, 2010.

\bibitem{L83}
T. Lyons. \newblock
A Simple Criterion for Transience of a Reversible Markov Chain. \newblock
\emph{Ann. Prob.} 11 (1983), no. 2, 393-402.

\bibitem{MS90}
L. De Michele and P. M. Soardi. \newblock 
A Thomson's Principle for Infinite, Nonlinear Resistive Networks. \newblock 
\emph{ Proc. Amer. Math. Soc.} 109 (1990), no. 2, 461-468.

\bibitem{M60}
G. J. Minty. \newblock
Monotone networks, \newblock
\emph{ Proc. Roy. Soc. London Ser. A} 257 (1960), 194-212. 

\bibitem{S94}
P.M. Soardi. \newblock
Potential Theory on Infinite Networks. \newblock
\emph{Lecture Notes in Math.}, 1590. Springer, Berlin, 1994.

\bibitem{P99}
Y. Peres. \newblock
Probability on trees: an introductory climb. \newblock
Lectures on probability theory and statistics (Saint-Flour, 1997), 193-280.
\emph{Lecture Notes in Math.}, 1717, Springer, Berlin, 1999.

\end{thebibliography}
\end{document}